\documentclass{amsart}
\usepackage{amssymb,amsmath}
\usepackage[latin1]{inputenc}

\newcommand{\tg}{{\ti g}}
\newcommand{\dist}{\mbox{\rm dist}}
\newcommand{\fall}{ \ \ \forall \ \ }

\newcommand{\Riem}{{\rm Riem}}

\newcommand{\vol}{{\rm vol}}

\newcommand{\Rb}{{\rm R}}

\newcommand{\norm}{ \|}

\newcommand{\C}{ \mathbb C}

\newcommand{\curlR}{ {\mathcal R} }

\newcommand{\ep}{\varepsilon}

\newcommand{\grad}{\nabla}

\newcommand{\si}{\sigma}

\newcommand{\parti}[2]{\frac{\partial #1 } {\partial #2} }
\newcommand{\lap}{\Delta}

\newcommand{\boundary}{\partial}

\newcommand{\AVR}{{\rm AVR}}

\newcommand{\on}{\over}

\newcommand{\ti}{\tilde}

\newcommand{\partt}{ {\partial \on {\partial t}} }

\newcommand{\Sc}{ { \rm R }}

\newcommand{\br}{ {\begin{rema}}}
\newcommand{\bl}{ {\begin{lemm}} }
\newcommand{\er} { {\end{rema}}}
\newcommand{\el}{{\end{lemm}} }
\newcommand{\bc}{  {\begin{coro}} }
\newcommand{\ec}{  {\end{coro}} }

\newcommand{\plusminus}{ \plusminus}

\newcommand{\de}{\delta}

\newcommand{\N}{\mathbb{N}}
\newcommand{\R}{\mathbb{R}}

\newcommand{\Ric}{\rm Ric}

\usepackage{amsthm} 
\usepackage{amsmath}
\usepackage{amssymb} 
\usepackage{latexsym} 
\usepackage{array} 
\usepackage[T1]{fontenc} %
\usepackage{fancyhdr} 
\usepackage{xspace} %
\usepackage[dvips]{epsfig} 
\usepackage{graphicx} 
 \usepackage[T1]{fontenc}
    \usepackage{ae}
    \usepackage{aecompl}

\newtheorem{defi}{Definition}[section]

\newtheorem{rema}[defi]{Remark}

\newtheorem{theo}[defi]{Theorem}
\newtheorem{lemm}[defi]{Lemma}

\newtheorem{coro}[defi]{Corollary}

\begin{document}
\title[Expanding solitons coming out of cones]{Expanding solitons with non-negative curvature\\ operator coming out of cones }

\author{Felix Schulze}
\address{Felix Schulze: 
  Freie Universit\"at Berlin, Arnimallee 3, 
  14195 Berlin, Germany}
\curraddr{}
\email{Felix.Schulze@math.fu-berlin.de}

\author{Miles Simon}
\address{Miles Simon: Universit\"at Magdeburg , Universit\"atsplatz 2,
39106 Magdeburg, Germany}
\curraddr{}
\email{msimon@ovgu.de}


\dedicatory{}


\maketitle 
 \numberwithin{equation}{section}
\numberwithin{defi}{section}

\maketitle  
\begin{abstract} We consider Ricci flow of
   complete Riemannian manifolds
  which have bounded non-negative curvature  
  operator, non-zero asymptotic volume ratio and no boundary.
 We prove scale invariant estimates for these solutions.
  Using these estimates, we show that there is a limit solution, obtained by  
  scaling down this solution at a fixed point in space. This limit solution  
  is an expanding soliton coming out of the   
  asymptotic cone at infinity. 
\end{abstract}  

\vskip  0.3 true in  

\section{Introduction and statement of results} 

\vskip  0.1 true in  

\noindent Let $(M^n,h)$ be a smooth  
$n$-dimensional, complete, non-compact Riemannian manifold without boundary,
with non-negative curvature  
operator and bounded curvature. In particular $(M,h)$ has  
non-negative sectional curvature and non-negative Ricci  
curvature. Any rescaling of this space also has  
non-negative sectional curvatures, and hence for every  
sequence of scalings  
$(M,c_i h,p_i)$, $c_i \in \R^+$  , $p_i \in M$, 
$i \in \N$, there exists a subsequence  
which converges   in the pointed Gromov Hausdorff  
sense to a metric space $(X,d_X)$ (see Proposition 10.7.1  
in \cite{BuBuIv}) which is a {\it metric space with curvature $\geq 0$}
(see Definition 4.6.2 in \cite{BuBuIv}).  \noindent In the case that
$c_i \rightarrow 0$ and $p_i  
= p$ for all $i \in \N$, the limit is known as $(X,d_X,0)$
{\it the asymptotic cone at infinity} and it is unique: 
see Lemma 3.4 of \cite{GK}. It is the Euclidean cone over a metric space   
$(V,d_V)$ where $(V,d_V)$ is an Alexandrov space of  
curvature bounded from below by one, and $0$ is the tip of the cone: see 
Corollary 3.5  of \cite{GK} for example. 

The Euclidean cone $CV$ over a metric space 
$V$ is homeomorphic to the space $\R^+_0 \times V/  
\sim $ with the quotient topology, where $(r,y) \sim   
(s,x)$ if and only if ($r=  s=  0$) or ($r= 
 s$ and $x= y$). The metric is given by: 
$$d_{CV}((r,x),(s,y)):=  r^2 + s^2 
-2rs\cos(\min(d_V(x,y),\pi)).$$ 
In the case that $CV$ arises as a Gromov Hausdorff limit   
in the setting described above,
and the sequence $(M,c_i h,p_i)$ is {\it non-collapsing},
then $V$ is homeomorphic to $S^{n-1}$. This follows from (unpublished) results
of G. Perelman \cite{PeUn}
as explained and simplified by V. Kapovitch in the paper
\cite{Kap1}. See Appendix \ref{B}. 

{\bf Notation:} A pointed 
sequence $(X_i,d_i,p_i)$ of metric spaces is {\it non-collapsing}
 if $\vol(B_1(p)) \geq \de >0$ for all $p \in X_i$ and some $\de >0$ independent
of $i$.

Our aim is to flow such cones $(CV,d_{CV})$ by Hamilton's Ricci flow
(introduced in \cite{Ha82}). 
We will show that a solution to Ricci flow with 
initial value given by the cone exists, that the 
solution is immediately smooth, and that it is an 
expanding Ricci soliton with non-negative curvature 
operator.
For an interval $I$ (open, closed, half open, finite or infinite length)
and a smooth manifold $M$ without boundary,
a smooth (in space and time) 
family of complete Riemannian metrics $\{g(\cdot,t) \}_{t \in I}$
solves {\bf Ricci flow} if
$ \partt g(t) = - 2\, \Ric(g(t))$ for all $t \in I$.
If $I =(0,T)$, we say $(M,g(t))_{t \in (0,T)}$ is a solution to Ricci flow
with initial value $(M,d_0)$ ($(M,d_0)$ a metric space), if 
 $\lim_{t \searrow 0} (M,d(g(t)))= (M,d_0) $ in the Gromov-Hausdorff sense
($(M,d(g))$ is the metric space associated to $(M,g)$).

\begin{defi} Let $(M,h)$ be a smooth, complete 
  Riemannian manifold without boundary, with non-negative Ricci curvature. The 
  asymptotic volume ratio $\AVR(M,h)$ is $$\AVR(M,h):=  
  \lim_{r \to \infty} \frac{\vol(B_r(x))} { r^n},$$
  where $x$ is an arbitrary point in $M$.
\end{defi} 
\noindent Due to the Bishop Gromov volume comparison 
principle, $\AVR(M,h)$ is well defined for such manifolds, and does
not depend on the point $x$. It also easily follows 
that  \begin{eqnarray} \frac {\vol(B_r(x))} { r^n} \geq 
  \AVR(M,h)\ \ \ \ \forall\ r >0. \end{eqnarray} In particular 
if $V_0:=   \AVR(M,h) >0$, then we have 
\begin{eqnarray} \frac {\vol(B_r(x))} { r^n} \geq V_0 >0 
  \ \  \ \ \forall\ r >0. \end{eqnarray}  
\begin{theo}\label{thm:mainthm} 
  Let $(M,h)$ be a smooth, 
  complete Riemannian manifold without boundary, with
  non-negative, bounded curvature operator and 
  positive asymptotic volume ratio $V_0 :=  \AVR(M,h)>0$. 
  Let $(X,d_X,0)$ be the asymptotic cone at infinity, i.e. the unique 
  Gromov-Hausdorff limit of $(M,c_i 
  h,p_0)$ for any sequence $c_i \to 0$ of positive numbers and any
  base point $p_0 \in M$. Then:
  \begin{itemize}
  \item[(i)] There exists a 
    smooth solution  $(M,g(t))_{t \in 
      [0,\infty)}$ to Ricci flow with $g(0)=  h$. 
  \item[(ii)] Let $g^i(t):= c_i g(t/c_i)$, $i \in \N$,
    be the solutions to Ricci flow obtained by
    rescaling the flows obtained in (i).
    The pointed solutions $(M,g_i(t),p_0)_{ t \in (0,\infty)}$
    converge smoothly, sub-sequentially (in the Hamilton-Cheeger-Gromov 
    sense: see \cite{Ha95a})
    as $i \to \infty$
    to a limit solution $(\tilde X,\tilde{g}(t)_{t \in (0,\infty)},\ti x_0)$.
    This solution  $(\tilde X,\tilde{g}(t)_{t \in (0,\infty)},\tilde x_0)$
    satisfies  $(\tilde{X},d(\tilde{g}(t)),\tilde x_0) \to (X,d_X,0)$ in  the 
    Gromov-Hausdorff sense as $t \searrow 0$, 
  $\text{AVR}(\ti X, \ti g(t)) = \text{AVR}(X,d_X) = \text{AVR}(M,h) \
    \ \forall\ t >0$,
  and  $X$ is homeomorphic to $\ti X$. 
    Hence such a limit solution may be thought
    of as a solution to Ricci flow with initial value $(X,d_X,0)$.
    Furthermore, 
    $(\tilde{X},\tilde{g}(t))_{t \in (0,\infty)}$ is an expanding 
    gradient  soliton with non-negative curvature operator. 
    That is $\tilde{g}(t)=  t (\phi_t)^*\tilde{g}(1)$ and 
    $\tilde{g}(1)$ satisfies \begin{equation} \Ric(\tilde{g}(1)) 
      -(1/2)\tilde{g}(1) + {}^{\ti g(1)}\grad^2 f \end{equation} for 
    some smooth function $f:M \to \R$. 
  \end{itemize}
\end{theo}

\begin{rema}
  As explained before, in this case 
  $X$ is the Euclidean cone over a metric space $V$,
  and $V$ is homeomorphic to $S^{n-1}$, see Appendix B. Thus $X$ is homeomorphic to $\R^n$ .
\end{rema}
\noindent 
In this paper (and in particular the above theorem) $ 
{}^g \grad^j$ refers to the $j$th covariant derivative 
with respect to $g$.  \noindent In order to prove 
this theorem we require a priori estimates for 
non-collapsed solutions with non-negative bounded curvature 
operator.
This involves proving a refined  version of Lemma 4.3 of \cite{Si6}, 
suited to the current setting, which we now state.

\begin{theo}\label{thm:longexist} 
  Let $(M^n,g_0)$ be a smooth, 
  complete Riemannian manifold without boundary,
  with bounded non-negative curvature 
  operator and $V_0:=   \AVR(M,g_0) >0$. Then there 
  exists a constant $c= c(n,V_0)>0$ and a solution to 
  Ricci-flow $(M,g(t))_{ t \in [0,\infty)}$ with $g(0)=  
  g_0$ such that 
  \begin{equation} 
    \sup_M |\Riem(g(t))| 
    \leq \frac c t\ ,   
  \end{equation}
  for all $t \in (0,\infty)$.
\end{theo}      

\noindent It is known that any smooth, open solution to Ricci flow with 
non-negative, bounded curvature operator has constant asymptotic
volume ratio, see \cite[Theorem 7]{Yokota08}. That is, in the above
theorem we have 
$$ V_t:= \textup{AVR}(M,g(t)) = V_0\, ,$$
for all $t\in[0,\infty)$.

After completing this paper a pre-print of E. Cabezas-Rivas and
B. Wilking appeared \cite{CabezasRivasWilking11} where the necessary a priori estimates are shown
to extend our results to open Riemannian manifolds with non-negative
(possibly unbounded)
complex sectional curvature . 

\section{Previous results and structure of the paper}
\noindent The literature that exists on expanding, shrinking and steady 
solitons is vast.
For a very good and current overview of the field, we refer the reader
to the survey paper of H-D.\,Cao \cite{Cao2}.
Here we mention some of the 
results on expanding solitons relevant to the  current setting.

In the paper \cite{Cao3}, the author constructs families of examples
of K\"ahler gradient expanding solitons on $\C^n$.
He also shows that any 
solution to K\"ahler Ricci flow with bounded curvature and which 
\begin{itemize}
\item[(i)] exists for $t \in (0, \infty)$, and
\item[(ii)] has non-negative holomorphic bi-sectional and positive
  Ricci curvature, and 
\item[(iii)] has $t \Sc(\cdot,t) \leq K$ for all $t>0$, and
\item[(iv)] $\sup_{(x,t) \in M \times (0,\infty)} t \Sc(x,t)$ is attained
\end{itemize}
must itself be an expanding K\"ahler gradient soliton.
This result was generalised to the case of Ricci flow with
non-negative curvature operator
by B-L.\,Chen and X-P.\,Zhu in \cite{ChZh} (see Proposition 4.2 there).
That is, if we have a solution
to Ricci flow which satisfies the above with 
condition $(ii)$ replaced by 
$(\ti{ii})$ {\it has non-negative curvature operator and positive
  Ricci curvature},
then the conclusion is, {\it the solution must be
  an expanding gradient soliton}.

Hence, it is natural to look for solutions satisfying
all or some of these conditions, when trying to construct expanding 
solitons with non-negative curvature operator.

Both of these theorems use the linear trace Harnack inequality of B.\,Chow
and R.\,Hamilton (see \cite{ChHa}).

A pre-print of L.\,Ma \cite{Ma} appeared after we had completed
this paper.
The pre-print contains a generalisation of the above 
result of B-L.\,Chen and X-P.\,Zhu \cite{ChZh}. In the paper \cite{Ma},  
the condition
$(\ti{ii})$ above is replaced by the new $(\ti{ii})$
 {\it has non-negative Ricci curvature and is non-collapsed}.
The paper of L.\,Ma uses the $W_+$ functional, which is a generalisation
of Perleman's $W$ energy (see \cite{Pe}).

In the papers \cite{FeIlNi} and \cite{NiKae}
the $W_+$ functional is studied globally and locally. 
$W_+$ is monotone non-decreasing, and constant precisely on expanders
(up to a shift in time). See for example Theorem 1.1 in \cite{FeIlNi} 
and also Corollary 6.9 of \cite{NiKae}, where a similar result
to that of \cite{ChZh} is proved using the $W_+$ functional.

It is known that if an expanding soliton has bounded curvature and
$\Ric > 0$, then $0< \AVR < \infty$.
This was first proved by R.\,Hamilton for $\Ric >0$ (see Proposition
9.46 in the book \cite{ChLuNi}).
In the paper \cite{CaNi} this result was generalised to the
case that the curvature is not necessarily bounded: see proposition
5.1 therein.
Sharper estimates for expanding gradient solitons under weaker assumptions
are also proved there. We refer the reader to that paper for more details.
Similar estimates may also be found in Proposition 4.1 of
the paper \cite{ChTa}.

Note that in our setting, we may assume that $\Ric>0$
after isometrically splitting
off a factor $\R^m$ (see Section \ref{Rescaling} for more details).
Hence, the assumption that the manifolds we consider have $ \AVR >0$
is natural.

Further examples of and estimates on
expanding,  steady and shrinking solitons on $\C^m$ are given
in \cite{FeIlKn}. In particular, they construct an example of a Ricci flow
which starts as a shrinking
soliton (for time less than zero), flows into  a cone at time zero and
then into a smooth expanding soliton (for time bigger than zero). 
They also include a discussion (with justification) on
the desirable properties of a {\it weak Ricci flow}.

The splitting result that we prove in Appendix A is essentially
derived from that of Hamilton in \cite{Ha86} (see also
\cite{Cao}).\\[1ex]
{\bf Structure of the paper:} In chapter three we fix some notation. In chapter
four we  prove a short time existence result for smooth, complete,
non-collapsed Riemannian manifolds with non-negative and bounded
curvature operator. In chapter five we show by a scaling argument that
these conditions actually imply longtime existence if the
asymptotic volume ratio is positive. Furthermore the asymptotic volume
ratio for the so obtained solutions remains constant. By blowing down
such a flow parabolically we prove in chapter six that we obtain a smooth
limiting solution, which evolves out of the asymptotic cone at infinity
of the initial manifold. We furthermore show that this solution
actually is an expanding soliton. In Appendix A we give a proof of a
splitting result, which has its origin in the de Rham Splitting
Theorem. In Appendix B we recall an approximation result from 
V.\,Kapovitch/G.\,Perelman.

\section{ Notation }
\noindent For a smooth Riemannian manifold $(M,g)$, and a           
family  $(M,g(t))_{t \in [0,T)}$ of smooth Riemannian metrics,
we use the notation
\begin{itemize} 
\item $(M,d(g))$ is the metric space associated to the Riemannian manifold
  $(M,g)$,
\item $d \mu_g$ is the volume form of the Riemannian manifold
  $(M,g)$,
\item $d(x,y,t)=  \dist_{g(t)}(x,y) = d(g(t))(x,y)$ is 
  the distance between $x$ and $y$ in $M$ with respect 
  to the metric $g(t)$, 
\item $ {}^g B_r(x)$ is the ball of radius $r$ and centre $x$ 
  measured with respect to 
  $d(g)$, 
\item $B_r(x,t)$ is the ball of 
  radius $r$  and centre $x$ measured with respect to $d(g(t))$, 
\item 
  $\vol(\Omega,g)$ is the volume of $\Omega$ with respect 
  to the metric $g$,
\item $\vol(  {}^g B_r(x))= \vol( {}^g B_r(x),g )$,
\item $\vol( B_r(x,t))= \vol(B_r(x,t),g(t))$,
\item  $\curlR(g)$ is the curvature 
  operator of $g$, 
\item $\Riem(g)$ is the curvature tensor of $g$,
\item $\Ric(g)$ is the Ricci curvature tensor of $g$,
\item $\Sc(g)$ is the scalar curvature of $g$,
\item $\Sc(p,g)$ is the scalar curvature of the metric $g$ at the point $p$.
\item If we write $B_r(x)$ resp. $\vol(B_r(x))$ then we mean
  ${}^g B_r(x)$ resp. $\vol(  {}^g B_r(x))$, where $g$ is a metric which 
  will be clear from the context.
\end{itemize}

\section{ Short time existence }\label{chapbound} \setcounter{equation}{0} 
\setcounter{defi}{0}  \numberwithin{defi}{section}  
\numberwithin{equation}{section} 
\noindent  Let $(M^n,g_0)$ be any 
smooth, complete  manifold with bounded curvature  and 
without boundary. From the results of R.\,Hamilton \cite{Ha82}
and  W-X.\,Shi \cite{Sh}, we know that there exists a 
solution $(M,g(t))_{t \in [0,T)}$ to Ricci flow with 
$g(0)=  g_0$ and $T \geq S(n,k_0)>0$ where $k_0 := \sup_M |\Riem(g_0)|$. 
That is: we can find a solution for a positive amount of time $T$
and $T$ is bounded from below by a constant
depending  on 
$k_0$ and $n$. The results of the paper \cite{Si6} show 
that if the initial manifold is smooth, complete, 
without boundary and has non-negative bounded curvature 
operator and  $\vol(B_1(x,0))  \geq v_0 >0$ for all 
$x \in M$, then  the there exists a solution for a 
time interval $[0,T)$ where $T\geq S(n,v_0)$. Note the 
difference to the results  of Hamilton and Shi: the 
lower bound on the length of the time interval of existence
does not depend on  the constant $k_0 :=  \sup_M |\Riem(g_0)| < \infty$.  Some 
estimates on the evolving curvature were also proved in 
that paper. We state this result here, and give a 
proof using the results of \cite{Si6}. 
\begin{theo}\label{Shorttime} 
  Let $(M,g_0)$ be smooth, 
  complete Riemannian manifold without boundary, with non-negative and bounded 
  curvature operator. Assume also that the manifold is 
  non-collapsed, that is \begin{eqnarray} \vol(B_1(x,0))  \geq 
    v_0 > 0 \fall x \in M. \end{eqnarray}  
  Then there 
  exist constants  $T=  T(n,v_0) >0$ and  $K(n,v_0)$ 
  and a solution to Ricci flow 
  $(M,g(t))_{t \in [0,T)}$ which satisfies 
  \begin{equation}\label{diogo} 
    \begin{split}   
      &(a_t) \ \  \curlR(g(t)) \geq 0, \\   
      &(b_t) \ \  \vol(B_1(x,t)) \geq {v_0 /2},  \\ 
      &(c_t) \ \  \sup_{M} |\Riem(g(t))| \leq K^2/t,  \\      
      &(d_t) \ \  d(p,q,s)  \geq  d(p,q,t)  \geq  
      d(p,q,s) -K      (\sqrt t - \sqrt s)   
    \end{split}
  \end{equation} for all $ x,p,q \in M$, $ 0 
  < s \leq t \in [0,T)$, where $\curlR(g)$ is the 
  curvature operator of $g$.
\end{theo}  
\begin{proof} 
  The proof follows from the results contained in the 
  paper \cite{Si6} and some other well known facts about 
  Ricci flow. Using the result of \cite{Sh}, we obtain a 
  maximal solution $(M,g(t))_{t \in [0,T_\text{max})}$ to Ricci flow 
  with $g(0)=  g_0$, where $T_\text{max}>0$ and $\sup_M 
  |\Riem(g(t))| < \infty$ for all $t \in [0,T_\text{max})$ and 
  $\lim_{t \nearrow T_\text{max}} \sup_M |\Riem(g(t))|=  \infty$ if 
  $T_\text{max} < \infty$. 
  Also the curvature operator of the solution is  
  non-negative at each time, since non-negative curvature 
  operator is preserved for solutions with bounded curvature 
  due to the maximum principle: see \cite{Ha95} and for 
  example the argument in Lemma 5.1 of \cite{Si6} (the 
  argument there shows that the maximum principle is 
  applicable to this non-compact setting, 
  in view of the fact that the curvature  is bounded. Maximum principles of 
  this sort are well known: see for example \cite{EH} or 
  \cite{NiTa2}). Now the result follows essentially by 
  following the proof of Theorem 6.1 in \cite{Si4}. For 
  convenience we sketch the argument here, and refer the 
  reader to the proof there for more details.  \noindent 
  Let $[0,T_M)$ be the maximal time interval for which 
  the flow exists and 
  \begin{eqnarray}
    \inf_{x \in M} \vol(B_1(x,t)) & > & {v_0 \on 2}, \label{maximal1} 
  \end{eqnarray} 
  for all $t \in [0,T_M)$. Using the 
  maximum principle and standard ODE estimates, one shows 
  easily that $T_M>0$ (see the proof of Theorem 7.1 in 
  \cite{Si6} for details). The aim is now to show that 
  $T_M \geq S$ for some $S=  S(n,V_0)>0$.  From 
  Lemma 4.3 of \cite{Si6} we see that if $T_M \geq 1$ 
  then the estimates  $(a_t)$, $(b_t)$ and $(c_t)$ are 
  satisfied for all $t \leq 1$, and $(d_t)$ would then follow
  from Lemma 6.1 of \cite{Si6}, and hence we would be finished. 
  So w.l.o.g. $T_M \leq 1$. From Lemma 4.3 of \cite{Si6} once again, 
  \begin{equation}\label{maximal2}
    |\Riem(g(t)| \leq {c_0(n,v_0)  
      \on t} 
  \end{equation} 
  for all $t \in (0,T_M)$ for some 
  $c_0=  c_0(n,v_0) < \infty$. First note that $(d_t)$ 
  holds on the interval $(0,T_M)$ in view of  Lemma 6.1 
  in \cite{Si6}, and the fact that  $\Ric(g(t)) \geq 
  0$. Using Corollary 6.2 of  \cite{Si6}, we see that 
  there exists an  $S=  
  S(v_0,c_0(v_0,n))=  S(n,v_0)>0$, such that $\vol(B_1(x,t)) 
  > {2v_0/3}$ for all  $t \in [0,T_M) \cap 
  [0,S)$. If $T_M < S$, then we obtain a contradiction 
  to the definition of $T_M$ ($T_M$ 
  is the first time where the condition  (\ref{maximal1}) is 
  violated). Hence $T_M \geq S$. But then we may use 
  Lemma 4.3, Lemma 6.1 of  \cite{Si6} to show that 
  $(a_t) , (b_t), (c_t)$ and $(d_t)$ are satisfied on $(0,S)$, as 
  required. 
\end{proof}   

\begin{rema}\label{helpful}
  Note that $T_\text{max} \geq (U(n) / k_0)$ and 
  $\sup_{M} |\Riem(g(t))| \leq k_0 \ti k(n)$ for
  all $t \leq (U(n) / k_0)$ for our solution,
  where $k_0 := \sup_{x \in M} |\Riem(g_0)| < \infty$ and $\ti k(n),U(n)> 0$
  are constants. This is due to the fact
  that our solution is constructed by extending a Shi solution,
  and the solutions of Shi satisfy such  estimates by scaling.
\end{rema}

\section{Long time existence and estimates} 
\setcounter{equation}{0}                 
                        
\noindent The long time existence result follows 
essentially from scaling. 
\begin{theo}\label{thm:longtimeexist} 
  Let $(M,g_0)$ be smooth, complete, without boundary,  with 
  non-negative bounded curvature operator. Assume also that 
  $\AVR(M,g_0) =:  V_0 > 0$.
  Then there exists a solution 
  to Ricci flow $(M,g(t))_{t \in [0,\infty)}$ with $g(0) = g_0$. Furthermore, 
  the solution satisfies the following estimates. 
  \begin{eqnarray} 
    (a_t) & \curlR(g(t)) \geq 0 \cr 
    (b^\prime_t)& \AVR(M,g(t)) = V_0 \cr \nonumber 
    (c_t)&\sup_{M} |\Riem(g(t))| \leq {K^2 \on t},  \cr \nonumber
    (d_t)& d(p,q,0)  \geq  d(p,q,t)  \geq  
    d(p,q,s) -K      (\sqrt t - \sqrt s) 
  \end{eqnarray} 
  for all $ t \in 
  [0,\infty)$ and $p,q \in M$,  where $K=  K(n,V_0)>0$ is a positive 
  constant and $\curlR(g)$ is the curvature operator of 
  $g$. 
\end{theo} 

\begin{proof} Let $c \in (0,\infty)$ and $\ti g_0 := c g_0$. Then we
  still have $\sup_M |\Riem(\ti g_0)| < \infty$ and $\AVR(M,\ti g_0)=
  V_0 >0$ as $\AVR(M,g)$ is a scale invariant quantity. From the
  Bishop-Gromov comparison principle, we have $\vol(\ti B_1(x)) \geq
  V_0 >0$ for all $x \in M$. Using the result above (Theorem
  \ref{Shorttime}), we obtain a solution $(M,\ti g(t))_{t \in
    [0,T(n,V_0))}$ satisfying $\ti g(0)= \tg_0 $ and the estimates
  $(a_{t}), (b_{t}), (c_{t}), (d_{t})$ for all $ t \in [0,T(n,V_0))$.
  Setting $g(t):= (1/c) {\ti g}(ct)$, for $t \in [0,(T/c))$ we also
  obtain a solution to Ricci flow with bounded non-negative curvature
  operator, satisfying $g(0)= g_0$ and the estimates $(a_t), (c_t) $
  and $(d_t)$ for all $t \in [0,(T/c))$, as $(a_t),(c_t),(d_t)$ are
  invariant under this scaling.  Furthermore, by \cite[Theorem
  7]{Yokota08} we have $\AVR(M,g(t)) = V_0$, and thus $\vol(B_r(x,t))
  \geq V_0 r^n$ for all $r>0$.  Now taking a sequence $c_i \to 0$ (in place
  of $c$ in the argument above), we obtain the result, in view of this
  estimate, $(a_t)$ and $(c_t)$, and the estimates of Shi and the
  compactness Theorem of Hamilton \cite{Ha95a}, see \cite{Ha95}.  Note
  that in fact $\sup_M |\Riem(g(t))| \leq k_0 \ti k(n) $ for all $t
  \leq U(n)/k_0$, where $k_0:= \sup_M |\Riem(g_0)|$, in view of Remark
  \ref{helpful}. Additionally $\sup_M |\Riem(g(t))| \leq \frac{K^2
    k_0}{U(n)}$ for all $t\geq U(n)/k_0 $ in view of the scale
  invariant estimate $(c_t)$ and hence the results of Shi (see
  \cite{Ha95}) apply.
\end{proof}

\section{Rescaling}\label{Rescaling} 
\setcounter{equation}{0}  

\noindent In this chapter we show that it is possible 
to scale down solutions of the type obtained in Theorem 
\ref{thm:longtimeexist} to obtain an expanding soliton  
coming out of the asymptotic cone $(X,d_X)$ at infinity of 
$(M,h)$.  
\begin{proof}[Proof of Theorem \ref{thm:mainthm}]  
  We assume that $(M,h)$ is a smooth manifold with 
  non-negative,   bounded curvature operator and positive 
  asymptotic volume ratio   $V_0:=  \text{AVR}(M,h)>0$. Now 
  let $c_i \rightarrow 0$   be a sequence of positive 
  numbers, converging to zero. Then $(M,c_ih, p_0)$   
  converges in the pointed  Gromov-Hausdorff sense to 
  the metric cone $(X,d_X,0)$.   By Theorem 
  \ref{thm:longtimeexist} there 
  exists a Ricci flow   
  $(M,g(t))_{t\in[0,\infty)}$ with $g(0)= h$ which satisfies  
  $(a_t),(c_t)$ and $\text{AVR}(M,g(t))= \text{AVR}(M,g(0))$.  
  By Hamilton's Harnack estimate, see for example equation 10.46 of
  Chapter 10, \S 4 of \cite{ChLuNi} , we have   
  \begin{equation}\label{eq:harnack}     
    \partt \big( \, t \, \Sc(p,g(t))\big) 
    \geq 0    
  \end{equation}  
  for all $t \in 
  [0,\infty)$ for all   $p \in M$.   We define the 
  scaled   Ricci flows $(M,g^i(t))_{t\in[0,\infty)}$ by   
  $$g^i(t):=  c_ig(t/c_i)\ .$$   Note that these flows 
  still satisfy $(a_t), (c_t)$ and $\text{AVR}(M,g^i(t)) =
  \text{AVR}(M,h)$ and we have a uniform lower bound for the
  injectivity radius for times $t \in [\delta, \infty)$, $\delta >0$,
  since the curvature is uniformly bounded and the volume of balls is
  uniformly bounded from below on such time intervals.  Hence, we may
  take a pointed limit of the flows $(M,g^i(t),p_0)_{t \in (0,
    \infty)}$ to obtain a smooth Ricci flow $(\ti X,\tilde{g}(t),\ti
  x_0)_{ t \in (0,\infty)}$ with
  \begin{eqnarray}
    \ti V_0 := \text{AVR}(\ti{X},\ti{g}(t)) 
    \geq  \text{AVR}(M,h)= V_0  > 0 \label{tvz}
  \end{eqnarray}
  for all $t \in (0,\infty)$.
  $\ti V_0 $ is a constant by \cite[Theorem
  7]{Yokota08}. 
Note that we also have the following estimates:
$d_X(\cdot,\cdot) \geq d(\ti g(t))(\cdot,\cdot) \geq d_X(\cdot,\cdot)
- K\sqrt t$ where $(X,d_X,0)$ is the asymptotic cone at infinity of
$(M,h)$.
These estimates follow after taking a limit of the estimates $(d_t)$
$d(c_ih)(\cdot,\cdot) \geq d(g^i(t))(\cdot,\cdot) \geq d(c_ih)(\cdot,\cdot)
- K\sqrt t$, which hold by construction of our solution.
In particular we see that $(\ti X, \ti g(t), \ti x_0)$ converges in the
pointed Gromov-Hausdorff sense to $(X,d_X,0)$ as $t \rightarrow 0$.

  A result of Cheeger and Colding (see Theorem 5.4  of \cite{ChCo}) gives that volume
  is continuous under the Gromov-Hausdorff limit of non-collapsing
  spaces with Ricci curvature bounded below. Thus since $(\tilde{X},
  \tilde{g}(t), \tilde{x}_0)$ converges to the asymptotic cone at infinity 
  $(X,d_X,0)$ of $(M,h,p_0)$ as $t\rightarrow 0$, and the
  Bishop-Gromov volume comparison
  principle holds, we have
  $$ \ti V_0\leq\text{AVR}(X,d_X)= \text{AVR}(M,h) = V_0\ ,$$
  and thus $\ti V_0 = V_0$. By \eqref{eq:harnack}, we also have
  $\partt (t \Sc(p,g^i(t)) ) \geq 0$ for all $t \in [0,\infty)$ 
  and all   $p \in M$.   For $p \in M$ define 
  $S(p) := \lim_{t\rightarrow \infty} t \Sc(p,g(t))$, which is 
  a well   defined and positive real (non-infinite) 
  number by   \eqref{eq:harnack} and $(c_t)$. Since this 
  quantity is scale-invariant it   follows that    
  $$\lim_{i\rightarrow \infty} t_0 \Sc(p,g^i (t_0))=  S(p)$$   
  for any fixed $t_0>0$. Note that the convergence of  
  $(M,g^i(t),p_0) \to (\ti X,\tilde{g}(t),\ti x_0)$ is smooth on 
  compact sets contained    in $(0,\infty) \times \ti X$, 
  and $p_0$ is mapped by the diffeomorphisms   involved 
  in the pointed Hamilton-Cheeger-Gromov convergence onto $\ti x_0$, 
  thus we have    $$ t\Sc(\ti x_0,\tilde{g}(t))=  S(p_0)$$ 
  for all $t>0$. Recall that the evolution equation for the Ricci
  curvature is given by
$$ \frac{\partial}{\partial t}\Ric^{i}_{\ j} = \Delta \Ric^{i}_{\ j}
+ 2\, \Ric^{r}_{\ s}\,\Riem^{i\ \ s}_{\ rj}\, .$$ Now note, that if
$\Ric(y,\tilde{g}(t))(Y,Y)= 0$ for some $Y \in T_y \ti X$ then we have
in view of the de Rham Decomposition Theorem (see Appendix \ref{A}
with $h$ of the Decomposition Theorem equal to $\Ric$) a splitting,
$(\ti X,\tilde{g}(t)) = (L \times \Omega, h \oplus l(t))$ where
$(L,h)$ has zero curvature operator and $(\Omega,l(t))$ has positive
Ricci curvature (here we use that $\Ric(Y,Y) = 0$ implies
$\sec(Y,V)=0$ for all $V$ in view of the fact that $\curlR \geq 0$).
In fact $(L,h) = (\R^k,h)$ where $h$ is the standard metric.  This may
be seen as follows.  If $(L,h)$ is not $ (\R^k,h)$, then the first
fundamental group of $(L,h)$ is non-trivial. This would imply in
particular that the first fundamental group of $(L \times \Omega, h
\oplus l(t)) = (\ti X,\tilde{g}(t))$ is also non-trivial. Using the
same argument given in Theorem 9.1 of \cite{Si6} (see Lemma
\ref{homeo} in this paper for some comments thereon), we see that
$(\ti X,\tilde{g}(t),\ti x_0)$ is homeomorphic to $(X,d_X,0)$, and
hence $(X,d_X)$ has non-trivial first fundamental group.  But as
explained at the beginning of this paper $(X,d_X)$ is a cone over a
standard sphere. In particular $(X,d_X)$ is homeomorphic to $\R^n$.
Hence $(X,d_X)$ has trivial first fundamental group which leads to a
contradiction.

Hence, we may write $(\ti X,\tilde{g}(t))= (\R^k \times \Omega, h
\oplus l(t))$ where $(\Omega,l(t))$ is a solution to the Ricci flow
satisfying $(a_t)$, $(c_t)$ and $\Ric(l(t)) >0$ for all $t>0$. Using
Fubini's theorem, it is easy to see that the asymptotic volume ratio
of $l(t)$ is given by $(\omega_{n-k}/\omega_{n}) \ti V_0 > 0$, where
$\omega_m$ is the volume of the $m$-dimensional Euclidean unit ball.

We show in the following that $(\Omega,l(t))_{t \in (0,\infty)}$ is a
gradient expanding soliton, generated by some smooth function
$f$. Hence $(\ti X,\tilde{g}(t))= (\R^k \times \Omega, h \oplus l(t))$
is a gradient expanding soliton: $ \grad^2 f(t) - \Ric(l(t)) -
(1/(2t)) l(t) = 0$ on $\Omega$ and $ \grad^2 v(t) -\Ric(h) - (1/(2t))h
= 0$ on $\R^k$ where $v(x,t) = \frac{|x|^2}{4t}$, and hence $ \grad^2
\ti f(t) - \Ric(\ti g(t)) - (1/(2t))\ti g(t) = 0$ for all $t>0$ on
$\ti X$ with
$\ti f(x,y,t)= f(y,t) + v(x,t)$ for $(x,y) \in (\R^k \times \Omega)$.\\[1ex]
For simplicity let us denote $l(t)$ again by $\tilde{g}(t)$, i.e. we
assume that $k= 0$.

\begin{rema} $(i)$ After completing this paper we noticed that we
  could use the proof of Proposition 12 of the paper of S.~Brendle
  \cite{SB} at this point to show that $(X,\ti g(t))$ is an expanding
  gradient soliton.  We include here our original proof
  which follows the lines of that given in \cite{ChLuNi}.\\[1ex]
  \noindent $(ii)$ In \cite{ChLuNi} and \cite{ChZh} it is assumed that
  $t \Sc(\cdot,t)$ achieves its maximum somewhere in order to conclude
  that the solution is a soliton. We make no such assumption.  We show
  that $\grad R (\ti x_0) = 0$ in view of the fact that $\partt (t
  R(t,\ti x_0))= 0$, and then argue as in \cite{ChLuNi} and
  \cite{ChZh}.
\end{rema}

\noindent For the rest of this argument we work with the Riemannian
metrics $\ti{g}(t)$. For ease of reading we introduce the notation
$\Sc(x,t):= \Sc(x, \tilde{g}(t))$, $\Ric(x,t):= \Ric(\ti g(t))(x)$ and
so on.  All metrics and covariant derivatives are taken with respect
to the metrics $\ti{g}(t)$.  We assume that $\ti g_{ij} = \de_{ij}$ at
points where we calculate, and indices that appear twice are
summed. We saw before that we may assume that $\partt(t \Sc(\ti x_0,
t ) )= 0$.  But then
\begin{eqnarray}\label{testy}
  0= \partt(t \Sc(\ti x_0,t))_{t=1}=  \lap \Sc(\ti x_0,1) +  
  2|\Ric|^2(\ti x_0,1) + \Sc(\ti x_0,1). \cr   
\end{eqnarray} 
By Theorem 10.46 in \cite{ChLuNi}, with $v_{ij}= \Ric_{ij}$, we have
\begin{equation} \label{eq:lintraceharnack} Z(Y):=
  \nabla_i\nabla_j\Ric_{ij}= |\Ric|^2 + 2(\nabla_j\Ric_{ij})Y_i
  +\Ric_{ij}Y_iY_j + \frac{\Rb}{2t} \geq 0\ ,
\end{equation} 
for any tangent vector $Y$.  In particular for $$Y=
-(\Ric^{-1})^{ji}\text{div}(\Ric)_j \parti{}{x^i}=
-(1/2)(\Ric^{-1})^{ji} \grad_j \Sc \parti{}{x^i},$$ we see that
\begin{equation}
  \label{eq:lintraceharnackv2} 
  \begin{split}      
    Z(Y)= &\ (1/2)\lap \Sc + |\Ric|^2 -
    (1/2)(\Ric^{-1})^{ij} (\nabla_i  \Sc) (\nabla_j  \Sc)\\
    &\ + (1/4)\Ric_{ij} (\Ric^{-1})^{si}\grad_s \Sc
    (\Ric^{-1})^{jk}\grad_k \Sc + \frac{\Rb}{2t}\\ = &\ (1/2)\Big(
    \lap \Sc +2 |\Ric|^2 - (1/2)(\Ric^{-1})^{ij}(\nabla_i
    \Sc)(\nabla_j \Sc) + \frac{\Rb}{t}\Big) \geq 0.
  \end{split}
\end{equation}   Using 
\eqref{testy}, we have 
\begin{equation}
  0 \leq Z(Y)(\ti x_0,1)=  - (1/4)(\Ric^{-1})^{ij}(\nabla_i  
  \Sc)(\nabla_j  \Sc)(\ti x_0,1) 
\end{equation}  
and hence $\nabla \Sc(\ti x_0,1)= 0$. This implies that
$$ Z(Y)(\ti x_0,1)=  0$$
which is a global minimum for $Z(Y)$. Now we use the evolution
equation for $Z(Y)$, which is given by equation (10.73) in
\cite{ChLuNi}.  It implies in particular that
\begin{eqnarray}
  \partt Z(Y) \geq \lap Z(Y), 
\end{eqnarray}  
and hence by the strong 
maximum principle, we must have   $Z(Y)=  0$ 
everywhere. By looking once again at  the 
equation (10.73) in   \cite{ChLuNi} and using   the 
matrix Harnack   inequality and the fact that $Z(Y) 
=  0$, we get   
\begin{equation*}    
  \partt Z(Y) \geq 2v_{ij}(\grad_k Y_i -\Ric_{ik} - (1/(2t))\de_{ik})(  
  \grad_k Y_j -\Ric_{jk} - (1/(2t))\de_{jk})\ .   
\end{equation*} 
If at some point in space and time we have $\grad Y - \Ric -
(1/(2t))\ti g \neq 0$ as a tensor, then we get $ 2v_{ij}(\grad_k Y_i
-\Sc_{ik} - (1/(2t))\de_{ik})( \grad_k Y_j -\Sc_{jk} -
(1/(2t))\de_{jk}) >0$ which would imply that $\partt Z(Y) >0$ at this
point in space and time. This in turn would imply that there are points in
space and time with $ Z(Y) > 0$, which is a contradiction.  Hence
$\grad Y - \Ric - (1/(2t))\ti g = 0$, which yields that $\ti g(t)$ is
an expanding gradient soliton.
\end{proof} 
\noindent For completeness we include the following Lemma, whose statement and
proof appeared in the proof we just gave.
\begin{lemm} \label{homeo}
  Let $(\ti X,\ti g(t),\ti x_0)_{t \in (0,\infty)}:=  \lim_{i \to 
    \infty}(M,g_i(t),p_0)_{t \in (0,\infty)}$   be the 
  solution obtained above, and $(X,d_X,0)$   be the 
  pointed Gromov-Hausdorff limit of $(M,c_ih,p_0)=   
  (M,g_i(0),p_0)$, ($c_i \to 0$).   
  Then $(\ti X,d(\ti g(t)),\ti x_0) \to (X,d_X,0)$ 
  in the pointed Gromov-Hausdorff sense  
  as $t \to 0$. That is, the solution flows out 
  of the cone   $(X,d_X,0)$. Furthermore, $\ti X$ is 
  homeomorphic to $X$ which is homeomorphic to $\R^n$. 
\end{lemm} 
\begin{proof} 
  (We repeat the proof given above).  Using the same argument given in
  Theorem 9.1 of \cite{Si6} , we see that $(\ti X,\tilde{g}(t),\ti
  x_0)$ is homeomorphic to $(X,d_X,0)$ ( note that in the argument
  there, $U$ and $V$ should be {\bf bounded} open sets: this is
  sufficient to conclude that the topologies are the same, since any
  open set can be written as the union of bounded open sets in a
  metric space).  But as explained at the beginning of this paper
  $(X,d_X)$ is a cone over a sphere, where the topology of the sphere
  is the same as that of the standard sphere. In particular $(X,d_X)$
  is homeomorphic to $\R^n$.
\end{proof}  

\begin{appendix} 
\begin{section}{de Rham splitting}\label{A} 
  \noindent In this appendix, we explain some known splitting results,
  which follow from the de Rham Splitting Theorem. We give proofs for
  the reader's convenience. We follow essentially the argument given
  in the proof of Theorem 2.1 of \cite{Cao} which follows closely that
  of Lemma 8.2/Theorem 8.3 of \cite{Ha86}. For a two tensor
  $\beta_{ij}$ we let $\beta^{i}_{\ j}:= g^{il}\beta_{lj}$, i.e. this
  defines an endomorphism from each tangent space into itself. 
\begin{theo} 
  Let ${h(t)}_{t \in [0,T)}$ be 
  a smooth (in space and time) bounded family of symmetric two 
  tensors
  defined on a simply connected  complete manifold $M^n$ without boundary,
  satisfying the evolution equation
  \begin{eqnarray} 
    \partt h^i_{\ j}=  {}^{g(t)}\lap h^i_{\ j} + \phi^i_{\ j}
  \end{eqnarray} 
  where $h(x,t), \phi(x,t) \geq 0$ for all $(x,t) \in M \times [0,T)$
  in the sense of matrices. We also assume $g$ is a smooth family of
  metrics (in space and time) satisfying $^{g_0}|D^{(i,j)} g| +
  {}^{g_0}|D^{(i,j)} h| \leq k(i,j)< \infty $ everywhere, where $i,j
  \in \N $ and $D^{(i,j)}$ refers to taking $i$ time derivatives and
  $j$ covariant derivatives with respect to $g_0$, and $k(i,j) \in \R$
  are constants.  Then for all $x \in M, t>0$, the null space of
  $h(x,t)$ is invariant under parallel translation and constant in
  time.  There is a splitting, $(M,g(t))= (N \times P, r(t) \oplus
  l(t))$, where $r,l$ are smooth families of Riemannian metrics such
  that $h>0$ on $N$ (as a two tensor), and $h = 0$ on $P$.
\end{theo} 
\begin{proof} 
  Let $0 \leq \si_1(x,t) \leq\si_2(x,t) \leq \ldots \leq \si_n(x,t)$
  be the eigenvalues of $h(x,t)$ and $\{ e_1(x,t) \}$ orthonormal
  eigenvectors. Assume that $\si_1(x_0,t_0) + \si_2(x_0,t_0) + \ldots
  + \si_k(x_0,t_0) >0$ at some point $x_0$ and some time $t_0$. Define
  a smooth function $\eta_{t_0}:M \to \R^+_0$ which is positive at
  $x_0$ and zero outside of $B_1(x_0,0)$ (the ball in $M$ of radius
  one with respect to $g_0$), and satisfies $\si_1(\cdot,t_0) +
  \si_2(\cdot,t_0) + \ldots + \si_k(\cdot,t_0) > \eta_{t_0}(\cdot).$
  Solve the Dirichlet problem:
  \begin{equation}
    \begin{split} &\partt \eta_i =    {} ^{g(t)} \lap 
      \eta_i \\
      &\eta_i(\cdot,t) |_{\boundary B_i(x_0,0)}=  0 
      \ \ \forall \ t \in [t_0,T) \\
      &\eta_i(\cdot,t_0)=  \eta_{t_0}(\cdot) 
    \end{split}
  \end{equation} 
  Using the estimates for $g$ we see that the solutions exist for all
  time and satisfy interior estimates independent of $i$ (see for
  example Theorem 10.1, chapter IV, \S 10 in \cite{LSU}). Thus we may
  take a subsequence to obtain a smooth solution $\eta: M \times
  [t_0,T) \to \R $ of the equation
  \begin{equation}
    \begin{split} 
      &\partt \eta=   {}^{g(t)} 
      \lap \eta \\
      &\eta(\cdot,t_0)=  \eta_{t_0}(\cdot)\ .
    \end{split}
  \end{equation}
  From the strong maximum principle, $\eta(\cdot,t) >0$ for all $t>
  t_0$. Also, the construction and the estimates on $g$ guarantee that
  $\sup_{(M \backslash B_i(x_0,0)) \times [t_0,S]} |\eta(\cdot,t)| \to
  0 $ as $i \to \infty$. for all $S < T$. We claim that
  $\si_1(\cdot,t) + \ldots + \si_k(\cdot,t) \geq e^{-at}\eta(\cdot,t)
  $ for all $t\geq t_0$. One proves first, that $\si_1(\cdot,t) +
  \ldots + \si_k(\cdot,t) - e^{-at}\eta(\cdot,t) + \ep
  e^{\rho^2(\cdot,t)(1+at) + at} \geq 0 $ for arbitrary small $\ep >0$
  and an appropriately chosen constant $a$, where here $\rho(x,t)=
  \dist(x,x_0,t)$ ($a>0 $ does not depend on $\ep$: $a$ depends on the
  constants in the statement of the Theorem). This is done by using
  the maximum principle. See for example the argument in the proof of
  Lemma 5.1 in \cite{Si6} for details. Now let $\ep $ go to zero. This
  implies $\si_1(\cdot,t) + \ldots + \si_k(\cdot,t) \geq \eta(\cdot,t)
  $ for all $t \geq t_0$ and hence $\si_1(\cdot,t) + \ldots +
  \si_k(\cdot,t) >0$ for all $t>t_0$. Thus 
  $$ dim (null (h(x,t)))=
  \max \{ i \in \{0, \ldots, n\} | \si_1(x,t) + \ldots + \si_i(x,t)= 0
  \}
  $$  
  is independent of $x \in M$ for all $t>t_0$ and a decreasing
  function in time. Hence $rank(h(x,t))$ is constant in space and
  time for some short time interval $ t_0 < t < t_0 + \de $ for any
  $t_0 \in [0,T)$. Now we let $v$ be a smooth vector field in space
  and time lying in the null space of $h$ (at each point in space and
  time). We can always construct such sections which have length one
  in a small neighbourhood, by defining it locally smoothly, and then
  multiplying by a cut-off function. We follow closely the proof of
  Lemma 8.2 of Hamilton (\cite{Ha86}) and Theorem 2.1 of
  \cite{Cao}. In the following we use the notation $\grad$ and $\lap$
  to refer to $ {}^{g(t)}\grad$ and $ {}^{g(t)}\lap$. Using $ h(v,v)
  \equiv 0$ we get
  \begin{equation}\label{ddth}
    \begin{split}
      0=\partt \Big(h(v,v)\Big) &=  \Big(\partt h^i_{\ j}\Big)v_iv^j +
      h^i_{\ j}\Big(\Big (\partt v_i\Big) v^j 
      + v_i\Big(\partt v^j\Big)\Big) \\
      &=  \Big(\partt h^i_{\ j}\Big)v^iv^j,
    \end{split}
  \end{equation} 
  since $h^i_{\ j}v_i= 0$ and $h^i_{\ j}v^j = 0$. Furthermore, since $
  h^i_{\ j} v_i v^j \equiv 0$ we get
  \begin{equation} \label{laph}
    \begin{split} 
      0 &= \lap( h^i_{\ j}  v_i  v^j) \\
      &=  (\lap h)^i_{\ j} v_i v^j + 2 \lap(v_i) 
      h^i_{\ j} v^j  \\
      &\ \   + 4 g^{kl} v^j\grad_k h^i_{\ j}  
      \grad_l v_i  +  2 g^{kl}  h^i_{\ j} \grad_k v^j 
      \grad_l v_i 
    \end{split}
  \end{equation} 
  The term $2 \lap(v_i) h^i_{\ j} v^j$ is once again zero, since
  $h^i_{\ j} v^j= 0$. Using this, \eqref{laph}, \eqref{ddth} and the
  evolution equation for $h$ we get
  \begin{equation}
    \begin{split}
      0 &=  \Big( \partt(h^i_{\ j}) -  (\lap h)^i_{\ j} - 
      \phi^i_{\ j}\Big)(v_iv^j)\\
      &=   4 g^{kl} v^j \grad_k h^i_{\ j}  \grad_l v_i  +  2 g^{kl}
      h^i_{\ j} 
      \grad_k v^j \grad_l v_i - \phi^i_{\ j}v_iv^j 
    \end{split}  
  \end{equation} 
  Now use  $$  v^j\grad_k h^i_{\ j}=  \grad_k(v^j 
  h^i_{\ j}) - h^i_{\ j} \grad_k v^j  = - h^i_{\ j} \grad_k v^j $$ to conclude  
  \begin{eqnarray} 2 g^{kl}  h^i_{\ j} 
    \grad_k v^j \grad_l v_i + \phi^i_{\ j}v_i v^j=  0 
    \label{maineqapp} \end{eqnarray} 
  Since $\phi(v,v) \geq 0$ (and 
  $h\geq 0$) we see that  $\phi^i_{\ j}v_i v^j=  0$. 
  That is, $v$ is also in the null space of $\phi$. 
  But then, \eqref{maineqapp} shows that $X_R(x,t): = \grad_{R} v(x,t)$ is 
  in the null space of  $h$ for any  vector $R \in T_x M $
  (choose orthonormal coordinates at $x$ at time $t$, so that 
  $\parti{}{x^1}(x):= R/\norm R \norm_{g(x,t)}$ and use this in
  equation \eqref{maineqapp}).  
  This shows that the 
  null space of $h$ is invariant under parallel transport 
  for each fixed time, as explained in the following for the
  readers convenience:
  \begin{itemize}
  \item[]
    Let $v_1(x), \ldots, v_k(x)$ be a smooth o.n.~basis for $null(h(x,t))$
    in a small spatial neighbourhood of $x_0$, and extend this to a
    smooth family $v_1, \ldots, v_n$ of vectors which is an o.n.~basis
    everywhere in a small spatial neighbourhood of $x_0$.
    Let $X_0\in T_{x_0}M$ satisfy $g(X_0,v_i(x_0)) = 0$ for all 
    $i \in \{k+1, \ldots, n\}$ and let $\gamma:[0,1]\rightarrow M$ be any
    smooth curve, starting in $x_0$ and whose image is contained in the
    neighbourhood of $x_0$ in question. Then parallel transport $X_0$  along
    $\gamma$. Call this vector field $X$.
    Write $X(\tau) = \sum_{i = 1}^n X^i(\tau)v_i(\gamma(\tau))$.
    We claim $ X(\tau)  = \sum_{i = 0}^{k} X^i(\tau) v_i(\gamma(\tau))$.
    Let $X^{\top}(\tau) =  \sum_{i = 1}^k X^i(\tau) v_i(\gamma(\tau)) $, and
    $X^{\perp}(\tau) =  \sum_{i = k+1}^n X^i(\tau) v_i(\gamma(\tau)) $.
    First note that for $i \in \{1, \ldots, k\}$, and $V$ the tangent
    vector field along $\gamma$:
    \begin{equation}
      g(\grad_V (X^{\perp}), v_i) 
      = V(g(X^{\perp}, v_i)) - g(X^{\perp}, \grad_V v_i)
      = 0
    \end{equation}
    in view of $ \grad_V v_i \in span\{ v_1, \ldots, v_k\}$
    and $X^{\perp}  \in span\{ v_{k +1}, \ldots, v_n\}$.
    Furthermore, for $j \in  \{ k+1, \ldots, n\}$ we have
    \begin{equation}
      \begin{split}
        g(\grad_V (X^{\perp}), v_j) &=  g(\grad_V (X- X^{\top}), v_j)\\
        &=  - g(\grad_V (X^{\top}), v_j)\\
        &= -g (\sum_{i=1}^k V(X^i)v_i,v_j) - g(\sum_{i=1}^k X^i \grad_V v_i,v_j)\\
        &= -\sum_{i=1}^k X^i g(\grad_V v_i,v_j)\\
        &= 0
      \end{split}
    \end{equation}
    in view of the fact that $\grad_V v_i \in  span\{ v_1, \ldots, v_k\}$.
    Hence $X^{\perp}$ is also parallel along $\gamma$. Since $X^\perp(0)=0$
    we have $X^\perp \equiv 0$.
  \end{itemize} 
  We have also shown that $ null(h) \subseteq null(\phi)$.  Let
  $v(x_0,s)$ for $s \in(t, t+\de)$ be smoothly dependent on time, and
  $ v(x_0,s) \in null(h(x_0,s))$ for each $s \in(t, t+\de)$.  Extend
  this vector at each time $s \in (t,t + \de)$ by parallel transport
  along geodesics emanating from $x_0$ to obtain a local smooth vector
  field $v(\cdot,\cdot)$ which satisfies $ v(x,s) \in null(h(x,s))$
  for all $x$ (in a small ball) and all $s \in (t,t+ \de)$. In
  particular,
$$\grad_iv \in null(h(x,s)) \ \ \text{and}\ \  
\lap v(x,s)=  (g^{kl}\grad_k \grad_l v)(x,s) \in null(h(x,s)).$$ 
Since $\grad_iv(x_0,s)=0$ we can compute
\begin{equation}
  \begin{split}  
    0=\partt \Big(h^i_{\ j} v_i\Big) &= \Big(h^i_{\ j} \partt v_i\Big)
    + \big(v_i (\lap
    h)^i_{\ j}\big) + v_i \phi^i_{\ j} \\
    &=  h^i_{\ j} \partt v_i  +  \lap (v_i h^i_{\ j}) +v_i\phi^i_{\ j} \\
    &= h^i_{\ j} \partt v_i
  \end{split}
\end{equation} 
where we have used that $ v \in 
null(\phi)$. Hence $\partt v(x_0,s) \in null(h(x_0,s))$. 
Assume that at time $s_0$ we have $null(h(x_0,s_0))=  
\R^k \subseteq \R^n=  T_{x_0} M$ and let $\{ e_1(t), 
\ldots, e_n(t) \}$ be a smooth (in time) o.n.~basis 
of vectors with $\{ e_1(t), \ldots, e_k(t) \}$ a 
smooth (in time) o.n.~basis of vectors of 
$null(h(x_0,t))$. Let $e_i^l(t):= \langle e_i(t), e^l(0)\rangle$,
where $\{ e^1(0), \ldots, e^n(0) \}$ refer to the standard basis vectors
of $\R^n$ and $\langle \cdot , \cdot \rangle$ is the standard inner
product on $\R^n$.
$\partt e_i(t) \in null(h(t))$ for all $i \in \{ 1, \ldots, k\}$ implies
$\partt e_i(t) = \sum_{j=  1}^k a_i^j(t)e_j(t)$ for some smooth functions
$a_i^j: [0, \infty) \to \R$, $i,j 
\in \{1, \ldots, k\}$. 
Then we have a system of ODEs ($l 
\in \{1, \ldots, n\}$, $i \in \{ 1, \ldots, k\}$) 
\begin{equation}
  \begin{split} 
    & \partt e_{i}^l (t)=  \sum_{j= 1}^k a_{i}^j(t)e_j^{l} (t) \\
    &  e_{i}^l (0)=  \de_i^l.  
  \end{split}
\end{equation} 
By assuming  $e_j^{l}(t)=  0$ 
for all $l \geq k +1$ we still have a solvable 
system, and hence the solution satisfies (by uniqueness) 
$e_j^{l}(t)=  0$ for all $l \geq k +1$. That is 
$\{ e_1(t), \ldots, e_k(t) \}$ remains in $\R^k$.   

$null(h(x_0,t))$ is a space which is invariant under 
parallel transport (from the argument above). Hence 
the de Rham splitting theorem (see \cite{deR})
says, $M $ splits 
{\bf isometrically at time $s$ } as  $ N(s) \oplus 
P(s)$ where $h(\cdot,s) = 0$ on $P(s)$ and $h(\cdot,s)>0$ on $N(s)$.
 We can do this it every time $s$. But the 
second part of the argument shows that $N(s)=  
N(s_0)$ for all $s$ and $P(s)=  P(s_0)$ for all 
$s$.  
\end{proof}  
\end{section}   

\begin{section}{An approximation result by V.\,Kapovitch/G.\,Perelman.}\label{B}
\noindent Let $(M_i,d_i,p_0)$ be a non-collapsing sequence of non-negatively
curved $n-$di\-men\-sio\-nal, smooth, complete manifolds without boundary
such that 
$(M^n_i,d_i,p_0) \to (X,d_X,0)$  as $i \to \infty$ 
(in the GH sense) where 
$X =CV$ is an Euclidean cone with non-negative curvature
over the metric space $(V,d_V)$ (with sectional curvature not less than 1 in
the sense of Alexandrov), and $(M_i,d_i,p_0)$ are smooth with
$\sec \geq 0$.
This is the situation examined in the introduction.
It is well known that the space of directions $\Sigma_{0}(X)$
of $(X,d_X)$ at $0$ is $(V,d_V)$: see Theorem 10.9.3
(here we have used that the tangent cone of $X$ at $0$ is equal to $X$, since
$X$ is a cone).
Now Theorem 5.1 of \cite{Kap1} says that
$\Sigma_{0}(X)$ is homeomorphic to $\Sigma_{p_0} M_i$  (for $i$ big enough)
which is isometric to the standard sphere $S^{n-1}$ since the $(M_i,g_i)$
are smooth manifolds.
That is $V$ is homeomorphic to $S^{n-1}$.
\end{section}
\end{appendix}    

\def\cprime{$'$} \def\cprime{$'$}
\providecommand{\bysame}{\leavevmode\hbox to3em{\hrulefill}\thinspace}
\providecommand{\MR}{\relax\ifhmode\unskip\space\fi MR }
\providecommand{\MRhref}[2]{%
  \href{http://www.ams.org/mathscinet-getitem?mr=#1}{#2}
}
\providecommand{\href}[2]{#2}


\begin{thebibliography}{10}

\bibitem{SB}
S.~Brendle, \emph{A generalization of hamilton's differential harnack
  inequality for the ricci flow}, J. Differential Geom. \textbf{82} (2009),
  207--227.

\bibitem{BuBuIv}
D.~Burago, Y.~Burago, and S.~Ivanov, \emph{A course in metric geometry},
  Graduate Studies in Mathematics, vol.~33, American Mathematical Society,
  Providence, RI, 2001.

\bibitem{CabezasRivasWilking11}
E.~Cabezas-Rivas and B.~Wilking, \emph{How to produce a ricci flow
  via cheeger-gromoll exhaustion}, {\tt arXiv:1107.0606v3}.

\bibitem{Cao3}
H-D. Cao, \emph{Limits of solutions to the {K}\"ahler-{R}icci flow}, J.
  Differential Geom. \textbf{45} (1997), no.~2, 257--272.

\bibitem{Cao}
H-D. Cao, \emph{On dimension reduction in the {K}\"ahler-{R}icci flow}, Comm.
  Anal. Geom. \textbf{12} (2004), no.~1-2, 305--320.

\bibitem{Cao2}
H-D. Cao, \emph{{Recent progress on {R}icci solitons}}, 2009, {\tt
  arXiv:0908.2006}.

\bibitem{CaNi}
J.~Carrilo and L.~Ni, \emph{Sharp logarithmic {S}obolev inequalities on
  gradient solitons and applications}, Comm. Ana. Geom. \textbf{17} (2010),
  no.~4, 1--33.

\bibitem{ChTa}
A.~Chau and L-F. Tam, \emph{On the simply connectedness of non-negatively
  curved {K}\"ahler manifolds and applications}, {\tt arXiv:0806.2457v1}.

\bibitem{ChCo}
J.~Cheeger and T.H. Colding, \emph{On the structure of spaces with {R}icci
  curvature bounded below. {I}}, J. Differential Geom. \textbf{46} (1997),
  no.~3, 406--480.

\bibitem{ChZh}
B-L. Chen and X-P. Zhu, \emph{Complete {R}iemannian manifolds with pointwise
  pinched curvature}, Invent. Math. \textbf{140} (2000), no.~2, 423--452.

\bibitem{ChHa}
B.~Chow and R.S. Hamilton, \emph{Constrained and linear {H}arnack inequalities
  for parabolic equations}, Invent. Math. \textbf{129} (1997), no.~2, 213--238.

\bibitem{ChLuNi}
B.~Chow, P.~Lu, and L.~Ni, \emph{Hamilton's {R}icci flow}, Graduate Studies in
  Mathematics, vol.~77, American Mathematical Society, Providence, RI, 2006.

\bibitem{deR}
G.~de~Rham, \emph{Sur la reductibilit\'e d'un espace de {R}iemann}, Comment.
  Math. Helv. \textbf{26} (1952), 328--344.

\bibitem{EH}
K.~Ecker and G.~Huisken, \emph{Interior estimates for hypersurfaces moving by
  mean curvature}, Invent. Math. \textbf{105} (1991), no.~3, 547--569.

\bibitem{FeIlKn}
M.~Feldman, T.~Ilmanen, and D.~Knopf, \emph{Rotationally symmetric shrinking
  and expanding gradient {K}\"ahler-{R}icci solitons}, J. Differential Geom.
  \textbf{65} (2003), no.~2, 169--209.

\bibitem{FeIlNi}
M.~Feldman, T.~Ilmanen, and L.~Ni, \emph{Entropy and reduced distance for
  {R}icci expanders}, J. Geom. Anal. \textbf{15} (2005), no.~1, 49--62.

\bibitem{GK}
L.~Guijarro and V.~Kapovitch, \emph{Restrictions on the geometry at infinity of
  nonnegatively curved manifolds}, Duke Math. J. \textbf{78} (1995), no.~2,
  257--276.

\bibitem{Ha82}
R.S.~Hamilton, \emph{Three-manifolds with positive {R}icci curvature}, J.
  Differential Geom. \textbf{17} (1982), no.~2, 255--306.

\bibitem{Ha86}
R.S.~Hamilton, \emph{Four-manifolds with positive curvature operator}, J.
  Differential Geom. \textbf{24} (1986), no.~2, 153--179.

\bibitem{Ha95a}
R.S.~Hamilton, \emph{A compactness property for solutions of the {R}icci flow}, Amer.
  J. Math. \textbf{117} (1995), no.~3, 545--572.

\bibitem{Ha95}
R.S.~Hamilton, \emph{The formation of singularities in the {R}icci flow}, Surveys in
  differential geometry, {V}ol.\ {II} ({C}ambridge, {MA}, 1993), Int. Press,
  Cambridge, MA, 1995, pp.~7--136.

\bibitem{Kap1}
V.~Kapovitch, \emph{Regularity of limits of noncollapsing sequences of
  manifolds}, Geom. Funct. Anal. \textbf{12} (2002), no.~1, 121--137.

\bibitem{LSU}
O.A.~Lady{\v{z}}enskaja, V.A. Solonnikov, and N.N. Ural{\cprime}ceva,
  \emph{Linear and quasilinear equations of parabolic type}, Translated from
  the Russian by S. Smith. Translations of Mathematical Monographs, Vol. 23,
  American Mathematical Society, Providence, R.I., 1967.

\bibitem{Ma}
M.~Li, \emph{Ricci expanders and type {III} {R}icci flow}, {\tt
  arXiv:1008.0711}.

\bibitem{NiKae}
L.~Ni, \emph{A new matrix {L}i-{Y}au-{H}amilton estimate for {K}\"ahler {R}icci
  flow}, {\tt arXiv:math/0502495v2}.

\bibitem{NiTa2}
L.~Ni and L-F. Tam, \emph{K\"ahler-{R}icci flow and the {P}oincar\'e-{L}elong
  equation}, Comm. Anal. Geom. \textbf{12} (2004), no.~1-2, 111--141.

\bibitem{PeUn}
G.~Perelman, \emph{Alexandrov spaces with curvatures bounded from below {II}},
  unpublished typed manuscript.

\bibitem{Pe}
G.~Perelman, \emph{The entropy formula for the {R}icci flow and its geometric
  applications}, {\tt arXiv:math/0211159}.

\bibitem{Sh}
W-X. Shi, \emph{Deforming the metric on complete {R}iemannian manifolds}, J.
  Differential Geom. \textbf{30} (1989), no.~1, 223--301.

\bibitem{Si6}
M.~Simon, \emph{Ricci flow of non-collapsed 3-manifolds whose {R}icci-curvature
  is bounded from below}, {\tt arXiv:0903.2142}.

\bibitem{Si4}
M.~Simon, \emph{Ricci flow of almost non-negatively curved three manifolds}, J.
  Reine Angew. Math. \textbf{630} (2009), 177--217.

\bibitem{Yokota08}
T.~ Yokota, \emph{Curvature integrals under the {R}icci flow on surfaces},
  Geom. Dedicata \textbf{133} (2008), 169--179.

\end{thebibliography}
\end{document}